\documentclass[11pt,reqno]{amsart}
\usepackage{amssymb}
\usepackage{mathptmx}
\usepackage{mathtools}
\usepackage{mathrsfs}
\usepackage[all]{xy}
\usepackage{stmaryrd}
\usepackage{fancyhdr}
\usepackage{hyperref}
\usepackage{mathrsfs}

\usepackage{pgf, tikz}
\usetikzlibrary{arrows, automata}

\usepackage{tikz-cd}
\usepackage{cite}
\usepackage{amsmath,amsfonts}
\usepackage{graphicx}
\usepackage{wrapfig}
\usepackage{array}
\usepackage{amsthm}
\usepackage{amsmath,amsthm,amssymb,hyperref}
\usepackage{exercise}

\usepackage{enumerate}
\usepackage{tikz}
\usepackage[left=1.2in, right=1.2in, top=1in, bottom=1in]{geometry}
\usepackage{etoolbox}
\usepackage{hyperref}
\patchcmd{\section}{\scshape}{\bfseries}{}{}
\makeatletter
\renewcommand{\@secnumfont}{\bfseries}
\makeatother
\xyoption{all}
\DeclareMathOperator{\Pic}{Pic}
\DeclareMathOperator{\Div}{Div}
\DeclareMathOperator{\dv}{Div}

\newcommand{\Jac}{\textrm{Jac}}{}
  
\newcommand{\Z}{\mathbb{Z}}

\input xy
\xyoption{all}
\theoremstyle{definition}
\newtheorem{mydef}{\textbf{Definition}}[section]
\newtheorem{myeg}[mydef]{\textbf{Example}}

\newtheorem{rmk}[mydef]{\textbf{Remark}}
\newtheorem*{que}{\textbf{Question}}

\theoremstyle{plain}
\newtheorem{mythm}[mydef]{\textbf{Theorem}}

\newtheorem*{nothma}{\textbf{Theorem A}}
\newtheorem*{nothmb}{\textbf{Theorem B}}
\newtheorem*{nothmc}{\textbf{Theorem C}}

\newtheorem{pro}[mydef]{\textbf{Proposition}}

\patchcmd{\abstract}{\scshape\abstractname}{\normalsize{\textbf{\abstractname}}}{}{}
\begin{document}

\title{Gluing of Graphs and Their Jacobians}

\author{Alessandro Chilelli}
\address{Department of Mathematical Sciences, State University of New York at Binghamton, NY 13902, USA}
\email{achilel1@binghamton.edu}

\author{Jaiung Jun}
\address{Department of Mathematics, State University of New York at New Paltz, NY 12561, USA}
\email{junj@newpaltz.edu}

\makeatletter
\@namedef{subjclassname@2020}{%
	\textup{2020} Mathematics Subject Classification}
\makeatother


\begin{abstract}
The Jacobian of a graph is a discrete analogue of the Jacobian of a Riemann surface. In this paper, we explore how Jacobians of graphs change when we glue two graphs along a common subgraph focusing on the case of cycle graphs. Then, we link the computation of Jacobians of graphs with cycle matrices. Finally, we prove that Tutte's rotor construction with his original example produces two graphs with isomorphic Jacobians when all involved graphs are planar. This answers the question posed by Clancy, Leake, and Payne in \cite{clancy2015note}, stating it is affirmative in this case. 
\end{abstract}

\maketitle

\section{Introduction}

A chip-firing game on a graph $G$ is a combinatorial game starting with a pile of chips (or ``negative'' chips) at each vertex of $G$. At each turn, a player chooses a vertex to lend (resp.~borrow) chips to (resp.~from) adjacent vertices. A player wins if no vertex has ``negative'' chips after finitely many turns. When one plays a chip-firing game, a natural question one may have is whether or not there is a winning strategy depending on an initial configuration of chips. 

In \cite{baker2007riemann}, inspired by the fact that finite graphs could be seen as a discrete analogue of Riemann surfaces, Baker and Norine formulated and proved an analogue of the Riemann-Roch theorem where effective divisors correspond to configurations where no vertex has negative chips. As an application, Baker and Norine provided an easy way to check whether or not there is a winning strategy for the chip-firing game in some cases. For instance, if the total number of initial chips on a graph is greater than or equal to the number $g=|E(G)|-|V(G)|+1$, then there is always a winning strategy for this initial configuration. For the precise statement, see \cite[Theorem 1.9]{baker2007riemann}. We note that Dhar's burning algorithm provides a resolution of a chip-firing game in general \cite[\S 3]{corry2018divisors}. In fact, a chip-firing game (more generally divisor theory for graphs) is one of the main tools in some sub-fields of algebraic geometry. We refer the interested reader to \cite{baker2016degeneration} for an extensive survey. 

Each configuration of chips on a graph $G$ can be considered as an element of the free abelian group $\mathbb{Z}[V(G)]$ generated by $V(G)$, the set  of vertices of $G$. Two configurations, $D_1$ and $D_2$, are equivalent if and only if $D_1$ can be obtained from $D_2$ by a finite sequence of moves (lending and borrowing). In particular, for a given initial configuration $D$, one has a winning strategy if and only if $D$ is equivalent to a configuration whose coefficient at each vertex is nonnegative. This defines a congruence relation $\sim$ on $\mathbb{Z}[V(G)]$, and hence we have the quotient group $\Pic(G):=\mathbb{Z}[V(G)]/\sim$, called the \emph{Picard group} of $G$. The \emph{Jacobian}  $\Jac(G)$ of $G$ is defined to be the torsion subgroup of $\Pic(G)$.


Due to its extensive applications, there is a growing interest in computing Jacobians for various families of graphs. In \cite{biggs1999chip}, Biggs computed the Jacobian of a wheel graph, $W_n$, when the number of vertices on the rim of $W_n$ is odd. Later, in \cite{norine2011jacobians}, Norine and Whalen computed Jacobians of nearly complete graphs and threshold graphs, and as an application they computed a remaining case of the Jacobian for wheel graphs. Also, Jacobians of iterated cones $G_n$ over a graph $G$, which is the join of $G$ and the complete graph $K_n$, have been studied in  \cite{brown2018chip}, \cite{goel2019critical}. We refer the reader to \cite{alfaro2012sandpile} and the references therein for an extensive list of authors contributing to this line of research. 

An interesting question one may ask is how the Jacobian of a graph changes under various graph operations such as deletion, contraction, or gluing along a common subgraph. For example, a nearly complete graph in \cite{norine2011jacobians} is a graph obtained by removing edges from a complete graph in a certain way. In general, based on our numerical experiments, it seems to be very hard to precisely compute how the Jacobian changes under these graph operations.

In this paper, we study the case where we glue two graphs along a common subgraph. Our motivation came from a question posed in \cite{clancy2015note}. The authors asked whether or not a certain graph-gluing process (Tutte's rotor construction) produces a pair of graphs whose Jacobians are isomorphic while they were proving another question concerning the two-variable zeta function of a graph.

To be precise, in \cite{lorenzini2012two} Lorenzini introduced the notion of a Riemann-Roch structure on a lattice of corank $1$ in $\mathbb{Z}^n$ (including the Riemann-Roch theory for graphs), and associated a two-variable zeta function to each such structure. Lorenzini's construction was inspired by several works on two-variable zeta functions for number fields and algebraic curves over finite fields \cite{deninger2003two}, \cite{lagarias2003two}, \cite{pellikaan1996special}, \cite{van2000effectivity}.\footnote{In the definitions of two-variable zeta functions above, $h^0(D)$ and $h^1(D)$ are not defined as the dimension of certain associated vector spaces. In \cite{borisov2003convolution}, Borisov constructed the spaces $H^0(D)$ and $H^1(D)$ which precisely compute $h^0(D)$ and $h^1(D)$ for the zeta function in \cite{van2000effectivity} by working in a larger category (than the category abelian groups).} With his construction, Lorenzini asked whether or not two connected graphs with the same Tutte polynomial should have the same associated two-variable zeta functions or isomorphic Jacobians, and he proved that for trees the answer is affirmative. 

In \cite{clancy2015note}, Clancy, Leake, and Payne proved that no two of these invariants determine the third in general. One of their methods was to use Tutte's rotor construction \cite{tutte1974codichromatic}. Roughly speaking, Tutte's rotor construction glues two graphs, $R$ and $S$, in two different ways through a fixed automorphism of $R$ (see \S \ref{section: Tutte's rotor construction} for the precise definition) producing two non-isomorphic graphs with the same Tutte polynomial. While they were producing counterexamples, they observed that applying Tutte's construction with his original example of a rotor of order 3 always produced a pair of graphs with isomorphic Jacobians in all of their test cases. Hence, they posed the following question:

\begin{que}(\cite[Question 1.4.]{clancy2015note})\label{question: Sam's paper}
	Does Tutte's rotor construction with his original example \cite[Figure 2]{tutte1974codichromatic} of a rotor of order 3 always generate a pair of graphs with isomorphic Jacobians?
\end{que}

To investigate the above question, we consider a more general situation of gluing two graphs along a common subgraph. Tutte's rotor construction corresponds to the case when a common subgraph is a set of isolated vertices. We note that the Jacobian of a graph $G$ can be computed by the Laplacian matrix of $G$. The gluing of Laplacians of graphs and their spectra has been studied in \cite{contreras2020gluing}, which might be useful for further investigation. \par\medskip

Let $G_1$ and $G_2$ be graphs with a common subgraph $H$, and $G=G_1\sqcup_HG_2$ be the graph obtained by gluing $G_1$ and $G_2$ along $H$. Unfortunately, there is no relationship among $\Jac(G)$, $\Jac(G_1)$, $\Jac(G_2)$, and $\Jac(H)$ in general even when $H$ is just an edge (see Example \ref{example: glued along one edge}). For this reason, we will mostly consider gluings of cycle graphs. 

In \S \ref{section: examples of jacobians under graph gluing}, we compute the Jacobian of several classes of graphs that we obtain by gluing cycle graphs. For example, we compute the Jacobian of the following gluing: Let $C_n$ and $C_k$ be cycle graphs with $n$ and $k$ vertices respectively, and $p$ be a positive integer less than $\min\{n,k\}$. Let $A$ (resp.~$B$) be a path of length $p$ in $C_n$ (resp.~$C_k$). Let $C_n*_{A,B} C_k$ be the graph obtained by gluing $C_n$ and $C_k$ along the paths $A$ and $B$ (see Example \ref{example: examples of gluing}). Then, we have the following.   


\begin{nothma}(Proposition \ref{proposition: gluing two consecutive})
With the same notation as above, we have
\[
\emph{Jac}(C_n*_{A,B} C_k) \simeq \mathbb{Z}/d\mathbb{Z} \times \mathbb{Z}/ ((nk-p^2)/d)\mathbb{Z}, 
\] 
where $d = \emph{gcd}(n,k,p)$.
\end{nothma}
We also apply a similar argument as in the proof of Theorem A to compute various gluings of graphs in \S \ref{section: examples of jacobians under graph gluing}.

For a graph $G$, its first (mod 2) homology $H_1(G,\mathbb{Z}_2)$ is called the cycle space (a vector space over $\mathbb{Z}_2$) of $G$.\footnote{Here, we consider $G$ as a simplicial complex.} The cycle space of $G$ can be considered as the set of all spanning Eulerian subgraphs of $G$, where addition is given by symmetric difference. A fundamental cycle of $G$ is a cycle created by adding an edge to a spanning tree of $G$. Once we fix a spanning tree, the set of fundamental cycles of $G$ forms a basis of $H_1(G,\mathbb{Z}_2)$. In \cite{chen2009critical}, Chen and Ye introduced the weighted fundamental circuits intersection matrix of a graph, and proved that it can be used to compute the Jacobian of a graph.

Let $G$ be a planar graph. We first introduce a matrix, $\mathbf{B}_G$, which is obtained from the face cycle matrix of $G$. Note that $\mathbf{B}_G$ is well-defined after we fix an embedding of $G$ into the plane. The matrix $\mathbf{B}_G$ encodes information of face cycles of $G$ and how they are adjacent to each other. In \S \ref{section: jacobians of graphs via cycle matrices}, we provide another way to compute the Jacobian in terms of face cycle matrices for planar graphs. This reduces the size of a matrix that we have to compute greatly. 

\begin{nothmb}(Proposition \ref{proposition: reduced = cycle matrix})
Let  $G$ be a connected, planar graph. Fix an embedding of $G$ into the plane. Then $\mathbf{B}_G$ and a reduced Laplacian $\widetilde{L}_G$ of $G$ have the same invariant factors. In particular, $\emph{Jac}(G)$ can be computed from $\mathbf{B}_G$.  
\end{nothmb}

\begin{rmk}
After we posted our first version on arXiv, we learned from Matt Baker the work \cite{chen2009critical} of Chen and Ye by which one can obtain our Theorem B as a special case. Also, we learned from Alfaro concerning his work with Villagr{\'a}n \cite{alfaro2020structure} which has the same definition as in our matrix $\mathbf{B}_G$.
\end{rmk}

Finally, in \S \ref{section: Tutte's rotor construction}, we answer \cite[Question 1.4]{clancy2015note} in the case when all involved graphs are planar. Our proof heavily depends on Theorem B by which we only have to keep track of newly-created face cycles (and how they are adjacent to other face cycles) after Tutte's rotor construction. We prove the following case of the question:

\begin{nothmc}(Theorem \ref{theorem: main theorem Tuttes for planar})
Let $R$ be Tutte's original example of a rotor of order $3$, and $S$ be a connected planar graph. Suppose that $G$ and $H$ are two graphs obtained from $S$ by Tutte's rotor construction with $R$. If $G$ and $H$ are planar, then $\Jac(G) \simeq \Jac(H)$. 
\end{nothmc}

\medskip
\textbf{Acknowledgment} We would like to thank Jaehoon Kim for many helpful conversations and his various comments on the first draft of the paper. We are grateful to Chris Eppolito for his detailed feedback and for pointing out some minor mistakes in the first draft. We also thank Yoav Len and Moshe Cohen for helpful comments on the first draft. We thank Matt Baker for pointing out the work of Chen and Ye \cite{chen2009critical} which is partially overlapped with ours. Finally, we thank Carlos A. Alfaro for letting us know his recent work and other related references. 

\section{Preliminaries}\label{section: preliminaries}

Throughout the paper, by a graph we always mean a finite, connected multigraph without loops unless otherwise stated. For a graph $G$, we let $V(G)$ be the set of vertices of $G$ and $E(G)$ be the multiset of edges of $G$. A \emph{divisor} on  $G$ is an element of the free abelian group generated by $V(G)$:
\[
\textrm{Div}(G)=\{\sum_{v \in V(G)} D(v)v \mid D(v) \in \mathbb{Z}\}.
\]
The \emph{degree} of a divisor $D=\sum_{v \in V(G)} D(v)v$, denoted by $\textrm{deg}(D)$, is the sum $\sum_{v \in V(G)} D(v)$. This defines the following group homomorphism:
\[
\textrm{deg}: \textrm{Div}(G) \to \mathbb{Z}, \quad D \mapsto \textrm{deg}(D). 
\]
Let $D_1=\sum_{v \in V(G)} D_1(v)v$ and $D_2=\sum_{v \in V(G)} D_2(v)v$ be divisors of $G$. We say that $D_1$ is obtained from $D_2$ by a \emph{lending move at $v$}, if
\begin{equation}
D_1=D_2-\sum_{vw \in E(G)} (v-w) = D_2 -\deg_G(v)v + \sum_{vw \in E(G)}w, 
\end{equation}
where $\deg_G(v)$ is the degree of a vertex $v$. Similarly, $D_1$ is obtained from $D_2$ by a \emph{borrowing move at $v$} if
\begin{equation}
D_1=D_2+\sum_{vw \in E(G)} (v-w) = D_2 +\deg_G(v)v - \sum_{vw \in E(G)}w.
\end{equation}
For $D_1,D_2 \in \Div(G)$, we let $D_1 \sim D_2$ if $D_1$ can be obtained from $D_2$ by a finite sequence of lending and borrowing moves. Clearly, this is a congruence relation on $\Div(G)$, hence we obtain the quotient group $\Pic(G):=\Div(G)/\sim$, called the Picard group of $G$. One can easily see that if $D_1 \sim D_2$, then $\deg(D_1)=\deg(D_2)$. In particular, the degree homomorphism factors through $\Pic(G)$, that is, we have the following homomorphism:
\begin{equation}\label{eq: deg map for Pic}
\deg: \Pic(G) \to \mathbb{Z}, \quad [D] \mapsto \deg(D), 
\end{equation}
where $[D]$ is the equivalence class of a divisor $D$. The Jacobian  of $G$, $\Jac(G)$, is the kernel of the degree homomorphism $\eqref{eq: deg map for Pic}$, and hence $\Jac(G)$ is the torsion subgroup of $\Pic(G)$. In fact, the following short exact sequence splits, where the set of sections of the homomorphism $\deg$ is in bijection with the subset $\Pic^1(G) \subseteq \Pic(G)$ consisting of the equivalence classes of degree $1$ divisors:
\begin{equation}
\begin{tikzcd}[column sep=0.7cm]
0 \arrow[r]& 
\Jac(G) \arrow[r] &
\Pic(G) \arrow[r, "\deg"] & 
\mathbb{Z} \arrow[r] &
0
\end{tikzcd}
\end{equation}

One may avoid using the combinatorial game of lending and borrowing moves and define the Jacobian of a graph purely in terms of linear algebra via the Laplacian of a graph. Recall that for a finite graph $G$, once we fix an ordering of $V(G)$, the Laplacian of $G$, $L_G$, is defined as follows:
\[
L_G= D_G - A_G,
\]
where $D_G$ is the degree matrix of $G$ and $A_G$ is the adjacency matrix of $G$. Then, one has the following map of $\mathbb{Z}$-modules:
\[
L_G:\mathbb{Z}^{|V(G)|} \to \mathbb{Z}^{|V(G)|}, \quad \vec{v} \mapsto L_G\vec{v}. 
\]
With this one has
\[
\Pic(G) \simeq \textrm{coker}(L_G). 
\]
Similarly, it is well-known that one can compute the Jacobian of a graph via the \emph{reduced Laplacian}, $\tilde{L}_G$, which is a matrix obtained by removing the $i$-th column and $i$-th row for any $i=1,2,\dots,|V(G)|$. Then, as in the case of the Laplacian, $\tilde{L}_G$ defines a $\mathbb{Z}$-module morphism, and 
\[
\Jac(G) \simeq \textrm{coker}(\tilde{L}_G). 
\] 
We note that one may compute the Smith normal form of a reduced Laplacian, $\tilde{L}_G$, to find the invariant factors, and hence find $\Jac(G)$. See \cite[\S 2]{corry2018divisors} or \cite{lorenzini2008smith}. 
 

\section{Jacobians Under Graph Gluing}\label{section: examples of jacobians under graph gluing}

In this section, we consider Jacobians of graphs obtained by gluing cycle graphs in several ways. To the best of our knowledge, the only consideration of Jacobians under graph gluing is when one glues two graphs along one vertex. This is elementary and well-known, but we include a proof for completeness. For instance, see \cite{corry2018divisors}.

\begin{pro}\label{proposition: gluing one vertex} 
Let $G_1$ and $G_2$ be graphs. Let $G$ be the graph obtained by gluing $G_1$ and $G_2$ along one vertex. Then 
\[
\emph{Jac}(G) \simeq \emph{Jac}(G_1) \times \emph{Jac}(G_2). 
\]
\end{pro}
\begin{proof}
Let $V(G_1)=\{v_1,\dots,v_n\}$ and $V(G_2)=\{u_1,\dots,u_m\}$. We may assume that $G$ is obtained by gluing $v_k$ and $u_k$. For a divisor $D$, we let $[D]$ be the divisor class of $D$ in $\Pic(G)$. For divisors $D_1=\sum_{i=1}^na_iv_i \in \dv(G_1)$ and $D_2=\sum_{i=1}^mb_iu_i \in \dv(G_2)$, we define the following divisor on $G$:
\[
D_1*D_2:=\sum_{i=1,i\ne k}^n a_iv_i + \sum_{i=1,i\ne k}^m b_iu_i + (a_k+b_k)v_k.
\]
Now, one can easily check that the following map is an isomorphism of groups:
\[
\Psi : \Jac(G_1) \times \Jac(G_2) \to \Jac(G), \quad ([D_1],[D_2])  \mapsto ([D_1*D_2]).
\]
\end{proof}

\begin{myeg}
Consider $G_1 = G_2 = K_3$. Let $G$ be the graph obtained by gluing $G_1$ and $G_2$ along one vertex - see Figure 1.
\begin{figure}[ht]
\centering
	  \includegraphics[width=.5\textwidth]{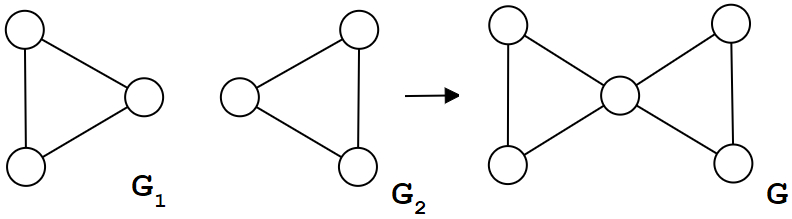}
\caption{\footnotesize$\ K_3$ glued to$\ K_3$ (single vertex)}\label{figure: single K_3}
\end{figure}

We have $\Jac(K_3) \simeq \mathbb{Z}/3\mathbb{Z}$. Hence, $\Jac(G_1) \times \Jac(G_2) \simeq \mathbb{Z}/3\mathbb{Z} \times \mathbb{Z}/3\mathbb{Z}$. A reduced Laplacian $\widetilde{L}_G$ and its Smith normal form $N_G$ are as follows:

\[
\widetilde{L}_G = \begin{bmatrix}{}
        2 & 0 & 0 & -1\\
        0 & 2 & -1 & 0 \\
        0 & -1 & 2 & 0 \\
        -1 & 0 & 0 & 2
    \end{bmatrix},
    \qquad 
  N_G=\begin{bmatrix}{}
  1 & 0 & 0 & 0\\
  0 & 1 & 0 & 0 \\
  0 & 0 & 3 & 0 \\
  0 & 0 & 0 & 3
  \end{bmatrix}
\]

It follows that $\Jac(G) \simeq \mathbb{Z}/3\mathbb{Z} \times \mathbb{Z}/3\mathbb{Z}$.
\end{myeg}

The following example shows that Proposition \ref{proposition: gluing one vertex} could fail to hold when we glue two graphs along even one edge. 

\begin{myeg}\label{example: glued along one edge}
Let $H_1 = H_2 = K_3$ and glue one common edge to obtain $H$ - see Figure 2.
    
\begin{figure}[ht]
   \begin{center}
	  \includegraphics[width=.5\textwidth]{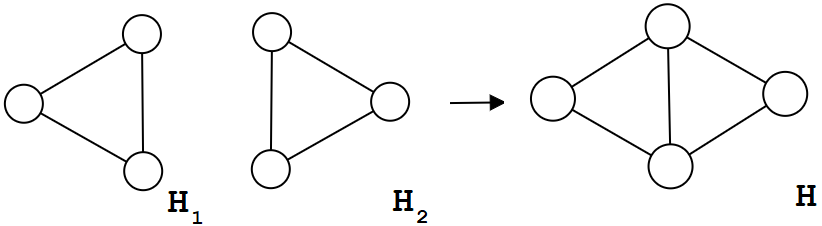}\\
	\end{center}
\caption{\footnotesize$\ K_3$ glued to$\ K_3$ (single edge)}
    \label{fig:my_label}
\end{figure}
A reduced Laplacian $\widetilde{L}_H$ and its Smith normal form $N_H$ are as follows:
\[
\widetilde{L} = \begin{bmatrix}{}
        2 & -1 & 0 \\
        -1 & 3 & -1 \\
        0 & -1 & 2 \\
    \end{bmatrix}, 
    \qquad 
N_H=\begin{bmatrix}{}
1 & 0 & 0\\
0 & 1 & 0 \\
0 & 0 & 8 \\
\end{bmatrix}    
\]
The invariant factors are $1, 1, 8$. This implies that $\ \Jac(H) \simeq \mathbb{Z}/8\mathbb{Z}$. In particular, 
\[
\Jac(H) \not \simeq \Jac(H_1) \times \Jac(H_2).
\]
This has failed Proposition \ref{proposition: gluing one vertex}. The special case of gluing along a single vertex is the key aspect and quite unique to the gluing process. This implies the resulting Jacobian of a resulting graph $H$ depends very highly on the way the two graphs $H_1$ and $H_2$ are arranged.
\end{myeg}

Let $C_n$ denote the cycle graph with $n$ vertices. Let $n,k,p$ be positive integers such that $p < \min\{n,k\}$, and $A$ (resp.~$B$) be an ordered set of edges of $C_n$ (resp.~$C_k$) such that $|A|=|B|=p$. We fix cyclic orientations of $C_n$ and $C_k$, and let $C_n*_{A,B} C_k$ be the graph obtained by gluing $C_n$ and $C_k$ along the edges in $A$ and $B$; if $A=\{e_1,\dots,e_p\}$ and $B=\{t_1,\dots,t_p\}$, then we glue $e_i$ and $t_i$ for each $i$ in such a way that the orientations of $e_i$ and $t_i$ are same. In general, the resulting graph $C_n*_{A,B} C_k$ does not have to be planar as it depends on how we glue $C_n$ and $C_k$. However, one can characterize gluing patterns of edges of $C_n$ and $C_k$ so that the resulting graph $C_n*_{A,B} C_k$ is planar by using the fact that a graph $G$ is planar if and only if the conflict graph of every cycle in $G$ is bipartite. See \cite{tutte1958homotopy}. 

We first compute the Jacobian of graphs obtained by gluing two cycle graphs along distinct edges where the resulting graph is planar. In the case of gluing one edge $(p=1)$, we obtain a graph whose Jacobian is cyclic. The following theorem by Cori and Rossin will be our main computational tool:

\begin{mythm}\cite[Theorem 2]{cori2000sandpile}\label{theorem: planar by Cori and Rossin}
Let $G$ be a planar graph and $\hat{G}$ be any of its duals~\footnote{We note that for a planar graph $G$, ``the'' dual graph $\hat{G}$ is not unique in the sense that it depends on a particular embedding.}, then
\[
\emph{Jac}(G)\simeq \emph{Jac} (\hat{G}).
\]
\end{mythm}

\begin{pro}\label{proposition: gluing two consecutive}
Let $n,k,p$ be positive integers such that $p <\min\{n,k\}$. Let $A=\{e_1,...e_p\}$ (resp.~$B=\{t_1,...,t_p\}$) be an ordered set of $p$-consecutive edges of $C_n$ (resp.~$C_k$). Then we have
 \[
  \emph{Jac}(C_n*_{A,B} C_k) \simeq \mathbb{Z}/d\mathbb{Z} \times \mathbb{Z}/ ((nk-p^2)/d)\mathbb{Z},
        \] \\
        where $d = \emph{gcd}(n,k,p)$. In particular, if $p=1$, then 
        \[
    \emph{Jac}(C_n*_{A,B} C_k) \simeq \mathbb{Z}/(nk-1)\mathbb{Z}.    
        \]
\end{pro}
\begin{proof}
Let $G=C_n*_{A,B} C_k$. Clearly $G$ is planar, so we let $\hat{G}$ be the planar dual of $G$. One can easily see that $\hat{G}$ has 3 vertices; 2 vertices to represent the planar regions contained by the cycle graphs, and a third vertex to represent the outer region. The first 2 vertices have exactly $p$ edges between them and $n-p$ and $k-p$ edges respectively to the third vertex. The Laplacian matrix $L_{\hat{G}}$ of $\hat{G}$ is given below.
\[
        L_{\hat{G}} = 
        \begin{bmatrix}
         n & -p & -(n-p) \\
         -p & k & -(k-p) \\
         -(n-p) & -(k-p) & n+k-2p \\
        \end{bmatrix}
\]
With respect to the third vertex, the reduced Laplacian matrix is given by the following matrix:
\[
   M= \begin{bmatrix}
         n & -p  \\
         -p & k \\
        \end{bmatrix}
\]
Now, one can easily check the Smith normal form of $M$ is as follows:
   \[
\begin{bmatrix}
d& 0  \\
0 & \frac{nk-p^2}{d} \\
\end{bmatrix}
\]
Therefore, from Theorem \cite{cori2000sandpile}, we obtain
\[
\Jac(G) \simeq \Jac(\hat{G}) \simeq \mathbb{Z}/d\mathbb{Z} \times \mathbb{Z}/ ((nk-p^2)/d)\mathbb{Z}.
\]
 \end{proof} 

\begin{myeg}\label{example: examples of gluing}
Let $G$ be the graph obtained by gluing $\ C_8$ and $\ C_{10}$ along 4 consecutive edges as in Figure 3 below.
 \begin{figure}[ht]\label{figure: C_10 glued to C_8 along 4 consecutive edges}
    \begin{center}
	  \includegraphics[width=.2\textwidth]{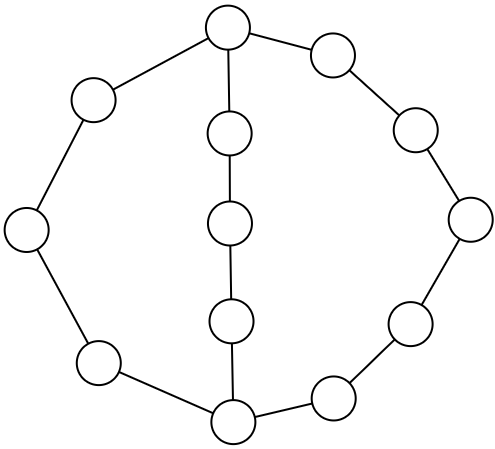}\\
	\end{center}
\caption{\footnotesize$\ C_8$ glued to$\ C_{10}$ (4 consecutive edges)}
\end{figure} 

We have $d=\text{gcd}(n,k,p) = \text{gcd}(10,8,4) = 2$ and $nk-p^2 = 64$. By Proposition \ref{proposition: gluing two consecutive}, we have 
\[
    \Jac(G) \simeq \mathbb{Z}/2\mathbb{Z} \times \mathbb{Z}/ 32\mathbb{Z}.
\]
\end{myeg}

In fact, one can apply the same idea as in Proposition \ref{proposition: gluing two consecutive} to the following gluing procedures. We omit the proofs. 

\begin{enumerate}
\item 
Let $C_n^*$ be a cycle graph where each edge splits into 2 distinct, undirected parallel edges. Then,
\begin{equation}
\Jac(C_n^*) \simeq \left(\Z / 2\Z \right)^{n-2} \times \Z / (2n) \Z.
\end{equation}
This may be seen as gluing two cycle graphs $C_n$ along the set of all isolated vertices.
\vspace{0.2cm}
\item 
\emph{Fan Graphs} are the join of $\overline{K_m}$ ($m$ vertices with no edge) and $P_{n}$ (a path with $n$ vertices), denoted as $F_{m,n}$. Here is $F_{1,5}$:
\begin{figure}[ht]
\begin{center}
	  \includegraphics[width=.2\textwidth]{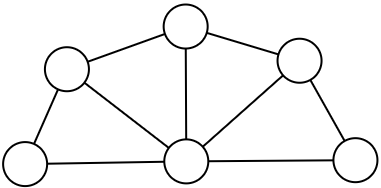}\\
	 \caption{\footnotesize Fan graph $F_{1,5}$}
	\end{center}
\end{figure}
	
One can easily compute that $\Jac(F_{1,5}) \simeq \Z/55\Z$. In general, one may apply the same idea as in the proof of Proposition \ref{proposition: gluing two consecutive} to obtain the following for $n \geq 3$:
\begin{equation}
\Jac(F_{1,n}) \simeq \Z/(3x_{n-1} - x_{n-2})\Z,
\end{equation}
where $x_i = |\Jac(F_{1,i})|$.
\vspace{0.2cm}
\item 
 Given a cycle graph $C_n$, let $A$ and $B$ be two disjoint sets of consecutive edges along $C_n$ each with $p_1$ and $p_2$ edges respectively. Let $p=p_1+p_2$, and suppose the $v_1 , v_{p_1}$ and  $v_{p_1+a} , v_{p+a}$ are the first and last vertices, respectively, of the paths $A$ and $B$. Draw an edge between $v_1$ and $v_{p+a}$, as well as $v_{p_1}$ and  $v_{p_1+a}$. Denote this new graph as $H$. Here is $C_{14}$ with $p_1 = 2$, $p_2=3$, and $a=5$ as an example in Figure 5:
\begin{figure}[ht]
\begin{center}	\includegraphics[width=.2\textwidth]{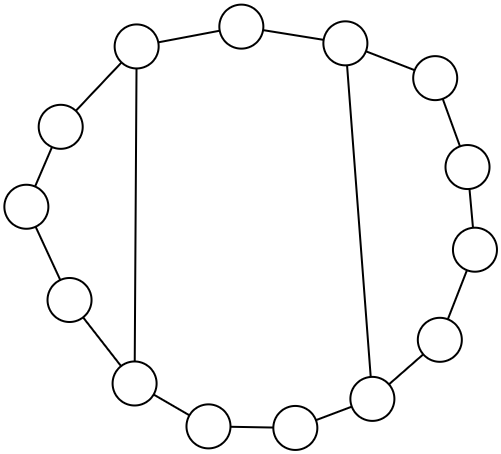}
\caption{\footnotesize Gluing of three cycles}
\end{center}	
\end{figure}    
   
Then, with $q(a) = n(p+1)-p^2+a(n-p)(p+2)-(p+2)a^2$, we have 
\begin{equation}
    \Jac(H) \simeq \Z/q(a)\Z.
\end{equation}
\vspace{0.2cm}
\item
   Given a cycle graph $C_n$, fix an independent set of 3 vertices, and add 3 edges joining them in a triangle. Let $a, b, c$ be the number of edges that $C_n$ has been partitioned into by this triangle, such that $a+b+c = n$. Denote this new graph by $H$, then we have
    \[
    \Jac(H) \simeq \Z/d\Z \times \Z/(f/d)\Z,
    \] 
    where $d = \text{gcd}(a+1,b+1,c+1)$ and
\[
f = n(a+1)(b+1)(c+1)- \big( a^2(b+1)(c+1) + b^2(a+1)(c+1)+c^2(a+1)(b+1) \big).
\]
The following is $C_{14}$ with $a = 4$, $b=6$, and $c=4$ in Figure 6:

\begin{figure}[ht]
\begin{center}
	  \includegraphics[width=.2\textwidth]{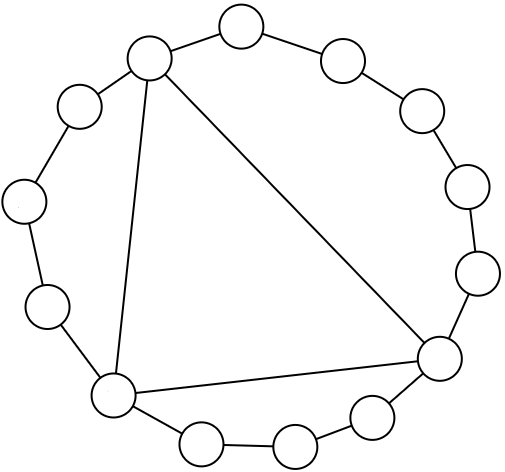}\\
	 \caption{\footnotesize Gluing of four cycles}
	\end{center}
\end{figure}	
\end{enumerate}

\section{The Jacobian of Graphs via Cycle Matrices}\label{section: jacobians of graphs via cycle matrices}

Let $G$ be a connected planar graph with $m$ edges and $n$ vertices containing $q$ cycles ($q\geq 1$). We assume that a planar graph $G$ is embedded into the plane. Fix an orientation on $G$, and choose some arbitrary positive direction of rotation, for instance, clockwise direction. For every edge of $G$, it will either be contained or not contained within a particular cycle as well as with or against this positive direction. Once we label the edges and cycles of $G$, we can define the \emph{cycle matrix},\footnote{Note that we slightly altered the definition to include orientation of a graph so that it would coincide with the reduced Laplacian later in Proposition \ref{proposition: reduced = cycle matrix}.} $B(G):=(B_{ij})_{q \times m}$ as follows:
\begin{equation}
B_{ij} = 
\begin{cases}
1, \quad \textrm{ if $i^{\textrm{~th}}$ cycle contains $j^{\textrm{~th}}$ edge in positive direction,} \\
-1, \quad \textrm{ if $i^{\textrm{~th}}$ cycle contains $j^{\textrm{~th}}$ edge in negative direction,} \\
0, \quad  \textrm{ otherwise.} \\
\end{cases}
\end{equation}
The rank of $B(G)$ is said to be the \emph{circuit rank}\footnote{This is also known as the genus $g(G)$ of a graph $G$.}, which is equal to $g=m-n+1$. 
By removing the rows from $B(G)$ which do not correspond to a face cycle, one obtains a $g\times m$ matrix $B(G)_f$. 
For notational convenience, we let
\[
\mathbf{B}_G:=B_f B_f^T, \quad B_f:=B(G)_f.
\]
One can easily check that $\mathbf{B}_G$ is a $g\times g$ symmetric matrix which is invertible, where $g=\textrm{rk}(B(G))$, and each entry of $\mathbf{B}_G$ is given as follows: 
\begin{equation}\label{eq: cycle matrix}
(\mathbf{B}_G)_{ij} = \begin{cases}
|f_i|, \textrm{ if $i=j$},\\
-|f_i \cap f_j|, \textrm{  if $i \neq j$}.\\
\end{cases}
\end{equation}
where $f_i$ is the $i^{th}$ face cycle and $|f_i|$ is the number of edges in $f_i$. In particular, this implies that $\mathbf{B}_G$ only depends on the underlying graph $G$ (without orientation or rotation), and hence $\mathbf{B}_G$ can be defined independently.\footnote{As we mentioned in introduction, after we posted our paper on arXiv, we learned that \cite{chen2009critical} and \cite{alfaro2020structure} have a similar construction.}

\begin{rmk}
As we mentioned, $\mathbf{B}_G$ depends on an embedding of $G$ into the plane. In what follows, we always assume that $G$ is embedded into the plane so that $\mathbf{B}_G$ is well-defined. 
\end{rmk}

\begin{myeg}
Let $H$ be the graph in Example \ref{example: glued along one edge}. Consider the following orientation on the graph with labeled edges as shown below. 

\begin{center}
\begin{figure}[ht]
	  \includegraphics[width=.2\textwidth]{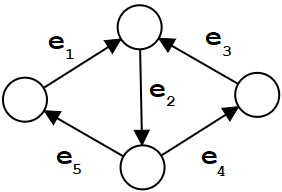}\\
	  \caption{\footnotesize $H$ with an orientation}
\end{figure}
	\end{center}

We consider a positive direction of rotation as being clockwise. We have exactly $3$ cycles, namely $C_1 = (e_1,e_2,e_5)$, $C_2=(e_2,e_3,e_4)$, and $C_3=(e_1,e_3,e_4,e_5)$. The cycle matrix $B(H)$ is given below.
\[
B(H)=	\begin{bmatrix}
		1 & 1 & 0 & 0 & 1 \\
		0 & -1 & -1 & -1 & 0 \\
		1 & 0 & -1 & -1 & 1
	\end{bmatrix}
\]
$C_1$ and $C_2$ are face cycles, and hence we have
\[
B_f=B(H)_f=	\begin{bmatrix}
1 & 1 & 0 & 0 & 1 \\
0 & -1 & -1 & -1 & 0 
\end{bmatrix}
\]
Hence we obtain
\[
\mathbf{B}_H=B_fB_f^T=	\begin{bmatrix}
3 & -1  \\
-1 & 3 
\end{bmatrix}
\]
\end{myeg}

Now, the following is straightforward. 

\begin{pro}
Let $G$ be a connected planar graph with $m$ edges and $n$ vertices containing $q$ cycles ($q\geq 1$). Then, with the same notation as above, the following hold. 
\begin{enumerate}
	\item 
All eigenvalues of $\mathbf{B}_G$ are positive. 	
	\item 
$B_fB_f^T$ and $B_f^TB_f$ have the same nonzero eigenvalues. 
\end{enumerate}
\end{pro}

One may compute the Jacobian of a planar graph $G$ via the matrix $\mathbf{B}_G$ as the following proposition shows. 

\begin{pro}\label{proposition: reduced = cycle matrix}
Let  $G$ be a connected, planar graph. Then $\mathbf{B}_G$ and a reduced Laplacian $\widetilde{L}_G$ have the same invariant factors. In particular, $\emph{Jac}(G)$ can be computed from $\mathbf{B}_G$.  
\end{pro}
\begin{proof}
Let $\hat{G}$ be the dual of $G$, and $\widetilde{L}_{\hat{G}}$ be the reduced Laplacian of $\hat{G}$ obtained by removing the row and column of the Laplacian of $\hat{G}$ corresponding to the ``exterior region'' of $G$. We claim that
\[
\mathbf{B}_G = \widetilde{L}_{\hat{G}}. 
\]	
In fact, suppose that vertices of $\hat{G}$ are labelled in such a way that the $i^{\textrm{~th}}$ face cycle of $G$ corresponds to the vertex $i$, and the ``exterior region'' of $G$ corresponds to the vertex $0$. One may observe that $(\widetilde{L}_{\hat{G}})_{ii}$ exactly counts the number of edges in the $i^{\textrm{~th}}$ face cycle. Also, for $i\neq j$, one can easily check that $(\widetilde{L}_{\hat{G}})_{ij}$ is the number of common edges between the $i^{\textrm{~th}}$ and $j^{\textrm{~th}}$ face cycles\footnote{Note this entry is counted as a negative number}. Now, our claim follows from the description \eqref{eq: cycle matrix} of $\mathbf{B}_G$, and the proposition follows from Theorem \ref{theorem: planar by Cori and Rossin} and the claim. 
\end{proof}

\begin{myeg}
Consider the following graph $G$:

\begin{center}
\begin{figure}[ht]
	\includegraphics[width=.3\textwidth]{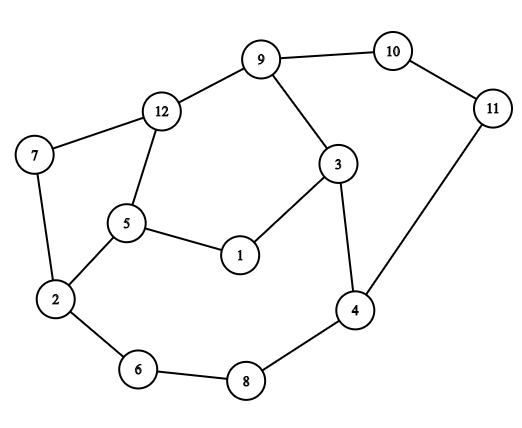}\\
		 \caption{\footnotesize Four face cycles}
\end{figure}
\end{center}

One may check that the Jacobian of $G$ is cyclic with order 476 by directly computing the Laplacian of $G$ which is of size $12 \times 12$. Much simpler, however, is the associated $\textbf{B}_G$ which is only $4\times4$ in size. It is given below.
\[
\mathbf{B}_G=\begin{bmatrix}
4 & -1 & 0 & -1 \\
-1 & 5 & -1 & -2 \\
0 & -1 & 5 & -1 \\
-1 & -2 & -1 & 7
\end{bmatrix}
\]
The Smith normal form of $\mathbf{B}_G$ is as follows:
\[
\begin{bmatrix}
		1 & 0 & 0 & 0 \\
		0 & 1 & 0 & 0 \\
		0 & 0 & 1 & 0 \\
		0 & 0 & 0 & 476
\end{bmatrix}
\]
Clearly, this agrees with the invariant factors produced by the Laplacian matrix of $G$. Thus $\textbf{B}_G$ has also given us $\Jac(G)\simeq \Z/476\Z$. 
\end{myeg}


We have the following proposition generalizing the fan graphs case $(2)$ in Example \ref{example: examples of gluing}. 

\begin{pro}\label{proposition: cycle jacobian}
Let $G$ be a graph obtained by gluing cycle graphs $C_{n_1},\ldots, C_{n_k}$, where each is glued along a single edge to the previous cycle and along a single edge to the next cycle. Then, $\emph{Jac}(G)$ is a cyclic group of order $x_k$, where $ x_i = n_ix_{i-1}-x_{i-2}$ for $3 \leq i \leq k$, and $x_1 = n_1$, $x_2 =n_1 n_2-1$.
\end{pro}
\begin{proof}
It follows from \cite[Theorem 2.4]{becker2016cyclic} that $\Jac(G)$ is cyclic. Hence it is enough to check the order of $\Jac(G)$. Since $G$ is a planar, connected graph, we can use $\mathbf{B}_G$ to compute the Jacobian by Proposition \ref{proposition: reduced = cycle matrix}. One can easily see that
\[
\mathbf{B}_G=
\begin{bmatrix}
n_1 & -1 & 0 & 0 & \cdots & 0 & 0 \\
-1 & n_2 & -1 & 0 & \cdots & 0 & 0 \\
0 & -1 & n_3 & -1 & \cdots & 0 & 0 \\
0 & 0 &  -1 & n_4 & \cdots & 0 & 0 \\
\vdots & & & & \ddots &  & \vdots \\
0 & 0 & 0 & 0 & \cdots & n_{k-1} & -1 \\
0 & 0 & 0 & 0 & \cdots & -1 & n_k \\
\end{bmatrix}
\]
Now, by induction, one can check that $\det(\mathbf{B}_G)=x_k$ as defined above. 
\end{proof}

\begin{myeg}
Consider a chain of cycles beginning with a 4 cycle, then attach a 6 cycle, then a 5 cycle, and finally a 3 cycle. The graphs are shown below.
\medskip

\hspace{-0.3cm}\includegraphics[width=.15\textwidth]{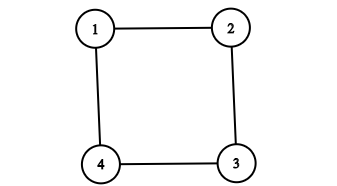} \includegraphics[width=.3\textwidth]{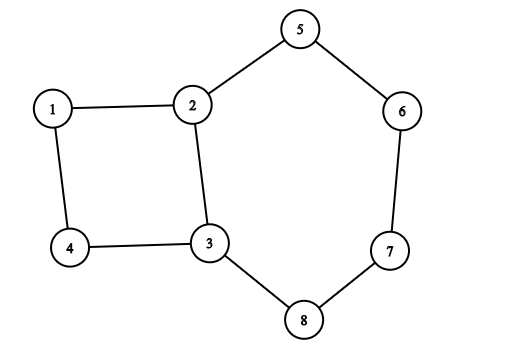}
\hspace{-0.4cm}
\includegraphics[width=.3\textwidth]{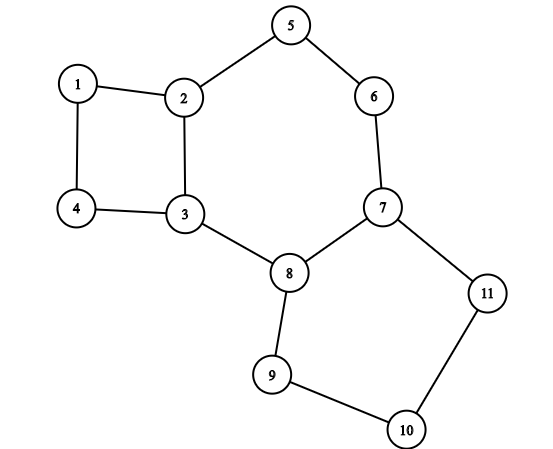}
\hspace{-0.2cm}
\includegraphics[width=.3\textwidth]{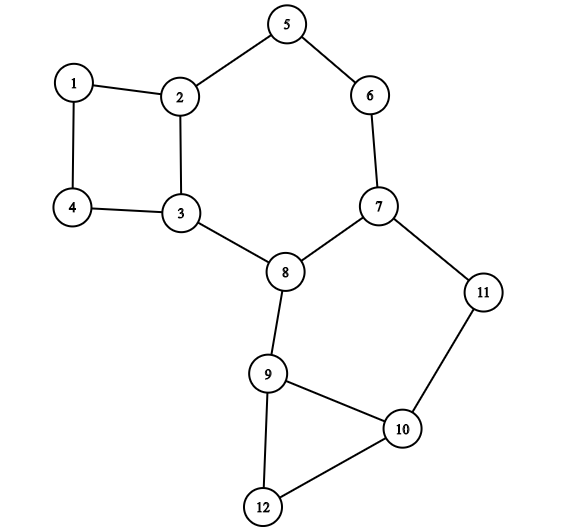}

$n_1 = 4$, $x_1 = 4$, \hspace{0.5cm} $n_2 = 6$, $x_2 = 23$,\hspace{2.2cm} 
$n_3 = 5$, $x_3 = 111$,\hspace{1.5cm} 
$n_4 = 3$, $x_4 =  310$.
\end{myeg}

\section{Tutte's Rotor Construction}  \label{section: Tutte's rotor construction}

We briefly recall Tutte's rotor construction in \cite{tutte1974codichromatic} which produces two non-isomorphic graphs with the same Tutte polynomial. A \emph{rotor} is a triple $(R,f,v)$ consisting of a graph $R$, a graph automorphism $f \in \textrm{Aut}(R)$ of order $n$, and a vertex $v \in V(R)$ such that $\#\{v,f(v),\dots,f^{n-1}(v)\}=n$. Let $S$ be another graph, and 
\[
g:\{v,f(v),\dots,f^{n-1}(v)\} \to V(S)
\]
be a function which does not have to be injective. Tutte's construction glues $R$ and $S$ in two different ways by using $f$ and $g$ to produce two new (non-isomorphic) graphs.\footnote{Tutte called $R$ the \emph{front-graph} and $S$ the \emph{back-graph}.} To be precise, the first glued graph is obtained by identifying $f^i(v)$ with $g(f^i(v))$. The second glued graph is obtained by identifying $f^i(v)$ with $g(f^{i+1}(v))$. We follow Tutte's notation and call the resulting graphs \emph{supergraphs}.

In this section, we prove the question \cite[Question 1.4]{clancy2015note} is true when the resulting supergraphs are planar. To this end, by $(R,f,v)$ we always mean Tutte's original example \cite[Figure 2]{tutte1974codichromatic}; $R$ is the graph in Figure $9$, $f$ is the automorphism of order $3$ such that $f(a)=b,f(b)=c,f(c)=a$, and $v=a$. By abuse of notation, we denote this rotor simply by $R$. 

\begin{figure}[ht]\label{figure: Tutte's original rotor}
\includegraphics[width=.3\textwidth]{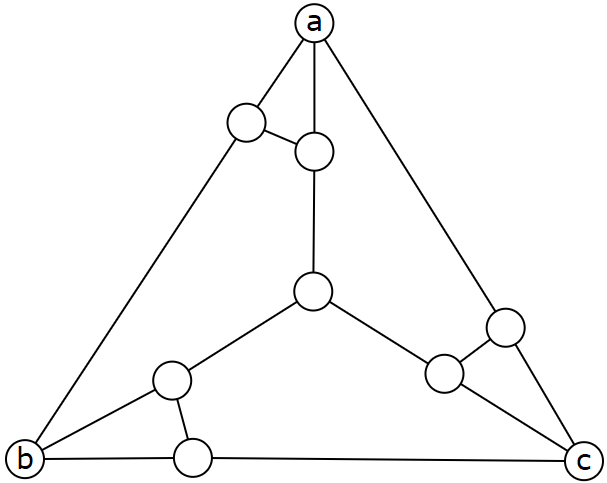}
\caption{Tutte's original example $R$}
\end{figure}

We first consider a \emph{variation of Tutte's construction} - we add an edge between two vertices. To be precise, with the same notation as above, the vertex $f^i(v)$ will be joined by an edge to $g(f^i(v))$ for $i=0,1,2$. This supergraph will be denoted by $G$. We will then use an automorphism $f : V(R) \to V(R)$ which will essentially reflect the graph $R$ along the center vertical line and then construct a new supergraph, $H$, by joining an edge between the vertex $f^{i}(v)$ and $g(f^{i+1}(v))$. Throughout this section, by abuse of notation, we let $G$ and $H$ be two supergraphs obtained in these two ways, although clearly they depend on the graph $S$ to which we glue.

We first prove that $\Jac(G)\simeq \Jac(H)$ when $S$ is a cycle by using the interpretation of the Jacobian via cycle matrices in \S \ref{section: jacobians of graphs via cycle matrices} as this proof will be modified to prove our main theorem and is more illustrating.

\begin{pro}\label{proposition: Tutte's variation}[Variation of Tutte's construction]
Let $S$ be a cycle graph. Let $G$ and $H$ be supergraphs obtained by the variation of Tutte's rotor construction explained above. Then, we have
\[
\Jac(G) \simeq \Jac(H). 
\]
\end{pro}
\begin{proof}
Let $R'=f(R)$. Since $R$ and $R'$ are isomorphic graphs, $\textbf{B}_R$ and $\textbf{B}_{R'}$ differ only by a permutation of rows and columns corresponding to relabeling face cycles. Let $n$ be the number of vertices of $S$. Choose $3$ vertices of $S$ (not necessarily distinct), and call them $x, y$ and $z$. Define $g$ as the following:
\[
g(a)=x, \quad g(b)=z, \quad g(c)=y. 
\]
With the construction defined above, consider the resulting supergraphs $G$ and $H$ below.\footnote{Although our picture describes when $g$ is injective, we do not assume that $g$ is injective.}
\begin{center}
	\includegraphics[width=.55\textwidth]{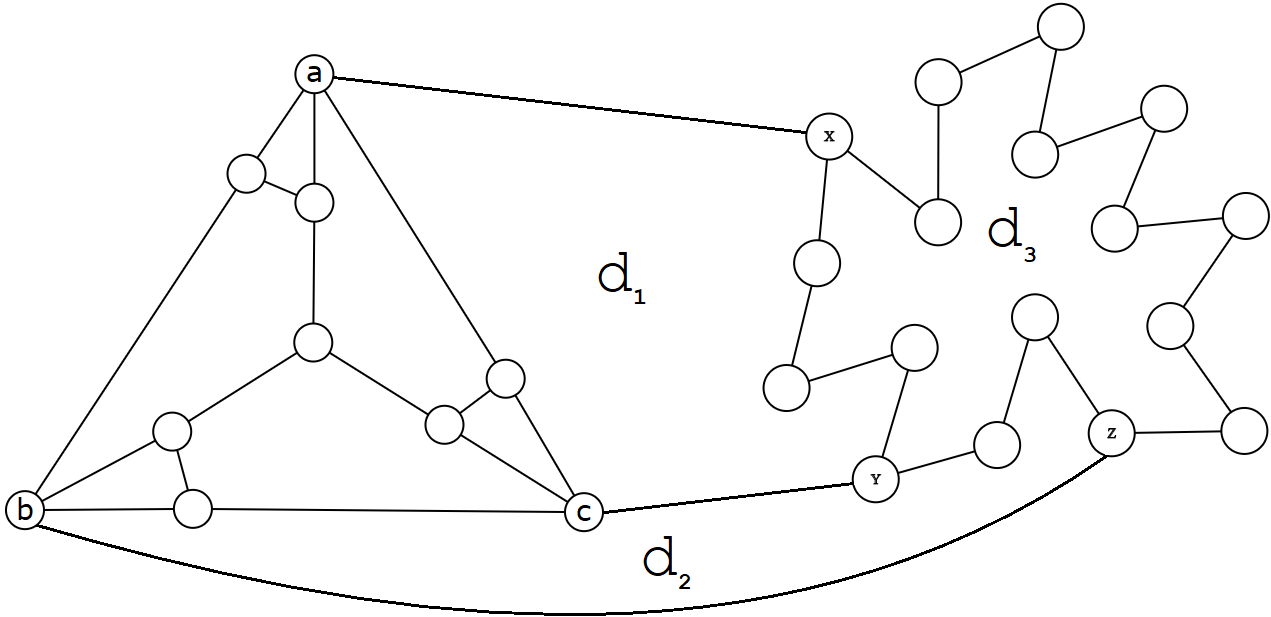}\\
	\textsc{Graph G}
\end{center}
\begin{center}
	\includegraphics[width=.55\textwidth]{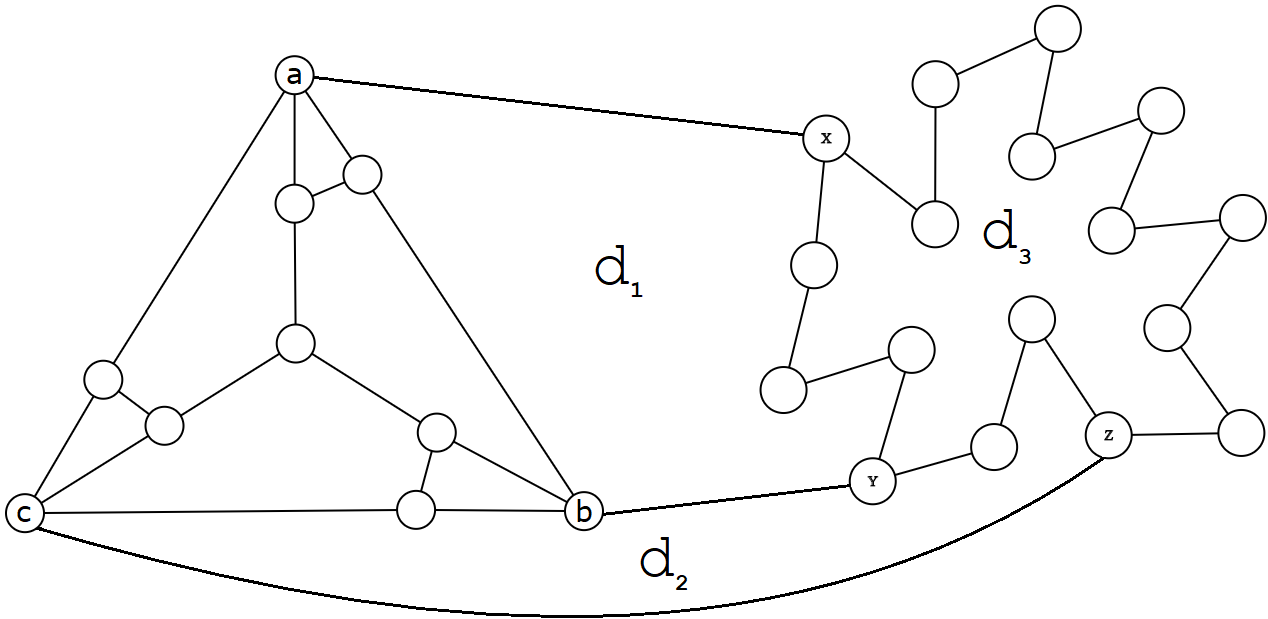}\\
	\textsc{Graph H}
\end{center}

Note that we immediately acquire 3 new face cycles labeled $d_1, d_2$, and $d_3$. Denote the number of edges between $d_1$ and $d_3$ by $|x-y|$ and the number of edges between $d_2$ and $d_3$ by $|y-z|$. Let $N$ be the subgraph of $G$ with the edges contained in only the new face cycles $d_1, d_2$, and $d_3$. Then, we have the following: 
\[
\textbf{B}_{N}=
	\begin{bmatrix}
		|d_1| & -1 & -|x-y| \\
		-1 & |d_2| & -|y-z| \\
		-|x-y| & -|y-z| & n
	\end{bmatrix}\\
\]

Note that $N$ is not altered in the reflection of $R$, that is $N$ is also the subgraph of $H$ obtained in the same way. In fact, $\textbf{B}_{N}$ is a submatrix of $\textbf{B}_G$ and $\textbf{B}_H$ as follows:
\[
	\textbf{B}_G=
	\begin{bmatrix}
		\textbf{B}_R & A_G  \\[3pt]
		A_G^T & \textbf{B}_{N} \\
	\end{bmatrix}, \qquad
	\textbf{B}_H=
	\begin{bmatrix}
		\textbf{B}_{R'} & A_H  \\[3pt]
		A_H^T & \textbf{B}_{N}  \\
	\end{bmatrix}
\]
where $A_G$ and $A_H$ are $6\times 3$ matrices which encode the adjacency between the face cycles of $R$ and $R'$ respectively and $N$ as follows:\footnote{$A_G$ and $A_H$ depend on the relabeling of face cycles.}

\[
	A_G=
	\begin{bmatrix}
		0 & 0 & 0\\
		0 & 0 & 0\\
		-1 & 0 & 0 \\
		-1 & 0 & 0 \\
		0 & -1 & 0 \\
		0 & -1 & 0 \\
	\end{bmatrix}, \qquad 
	A_H=
	\begin{bmatrix}
		-1 & 0 & 0 \\
		-1 & 0 & 0 \\
		0 & 0 & 0\\
		0 & 0 & 0\\
		0 & -1 & 0 \\
		0 & -1 & 0 \\
	\end{bmatrix}
\]
\vspace{0.1cm}

Clearly, these matrices differ by $2$ row permutations corresponding to the relabeling of face cycles of graphs $R$ and $R'$. This implies that $\textbf{B}_G$ and $\textbf{B}_H$ only differ by $2$ row and column permutations which will only affect the submatrices $\textbf{B}_R$ and $\textbf{B}_{R'}$. In particular, $\mathbf{B}_G$ and $\mathbf{B}_H$ have the same Smith normal form. Now, from Proposition \ref{proposition: reduced = cycle matrix}, we have $\Jac(G) \simeq \Jac(H)$.
\end{proof}

\begin{rmk}
One may observe that our proof of Proposition \ref{proposition: Tutte's variation} can be modified to prove a more general result when $S$ is a planar graph.
\end{rmk}

Now we turn our attention to Tutte's original construction; we glue the vertices $a,b,c$ to $g(a)$, $g(b)$, $g(c)$ respectively by identifying each as a single vertex without adding any new edges. We prove that Tutte's original construction with a connected planar graph $S$ obtains two resulting supergraphs which have isomorphic Jacobians if they are planar.

\begin{mythm}\label{theorem: main theorem Tuttes for planar}[Tutte's original construction]
Let $S$ be a connected planar graph. With the same notation as above, if the supergraphs $G$ and $H$ are planar, 
then we have $\Jac(G) \simeq \Jac(H)$. 
\end{mythm}
\begin{proof}
One can observe that $\textbf{B}_S$ is a submatrix in both $\textbf{B}_G$ and $\textbf{B}_H$. Similar to the proof of Proposition \ref{proposition: Tutte's variation}, we let $N$ be the subgraph of $G$ with the edges contained in the new face cycles created after gluing. One can easily check that $N$ does not change after reflecting $R$, that is, $N$ is also the subgraph of $H$ with the edges contained in the new face cycles created after gluing. Hence, $\textbf{B}_{N}$ is a submatrix in both $\textbf{B}_G$ and $\textbf{B}_H$. In fact, we have the following:

\[
\textbf{B}_G=
	\begin{bmatrix}
		\textbf{B}_R & A_G & \textbf{0}  \\[3pt]
		
		A^T_G & \textbf{B}_{N} & C_S  \\[3pt]
		\textbf{0}^T & C^T_S & \textbf{B}_S
	\end{bmatrix}, \qquad
	\textbf{B}_H=
	\begin{bmatrix}
		\textbf{B}_{R'} & A_H & \textbf{0}  \\[3pt]
		A^T_H & \textbf{B}_{N} & C_S  \\[3pt]
		\textbf{0}^T & C^T_S & \textbf{B}_S
	\end{bmatrix}
\]
\vspace{0.2cm}

\noindent where $C_S$ encodes the adjacency between the face cycles of $S$ and $N$ which are the same for both $G$ and $H$. $A_G$ and $A_H$ are the same as in Proposition \ref{proposition: Tutte's variation} which also only differ by permutation. It follows that $\textbf{B}_G$ can be obtained from $\textbf{B}_H$ via row and column permutations. In particular, $\textbf{B}_G$ and $\textbf{B}_H$ have the same Smith normal form, and hence $\Jac(G) \simeq \Jac(H)$ by Proposition \ref{proposition: reduced = cycle matrix}. 
\end{proof}

\begin{myeg}
Consider $S$ to be the following graph: 
\begin{center}
	\includegraphics[width=.5\linewidth]{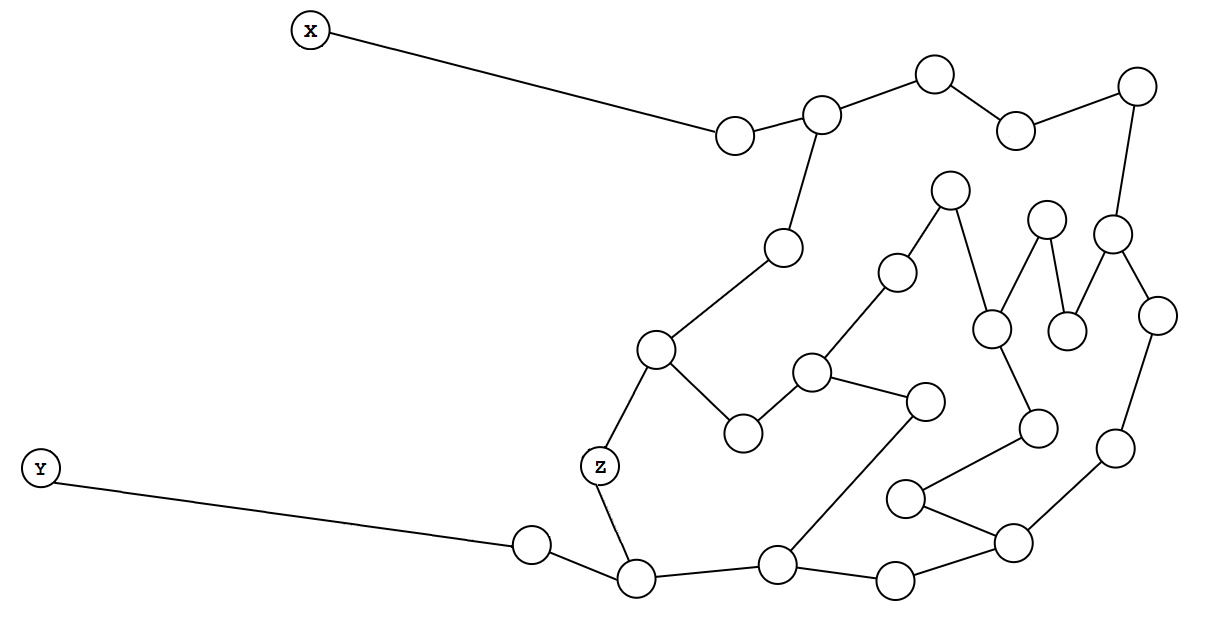}
\end{center}

Defining $g$ as $g(a)=x , g(b)=y , $ and $g(c)=z$, the constructions for graphs $G$ and $H$ are below. One can check the Smith normal form of $\textbf{B}_G$ and $\textbf{B}_H$ are the same and $\Jac(G)\simeq \Jac(H)$ is a cyclic group of order 163,780,565.	

\begin{figure}[ht]
			\begin{center}
				\includegraphics[width=.4\textwidth]{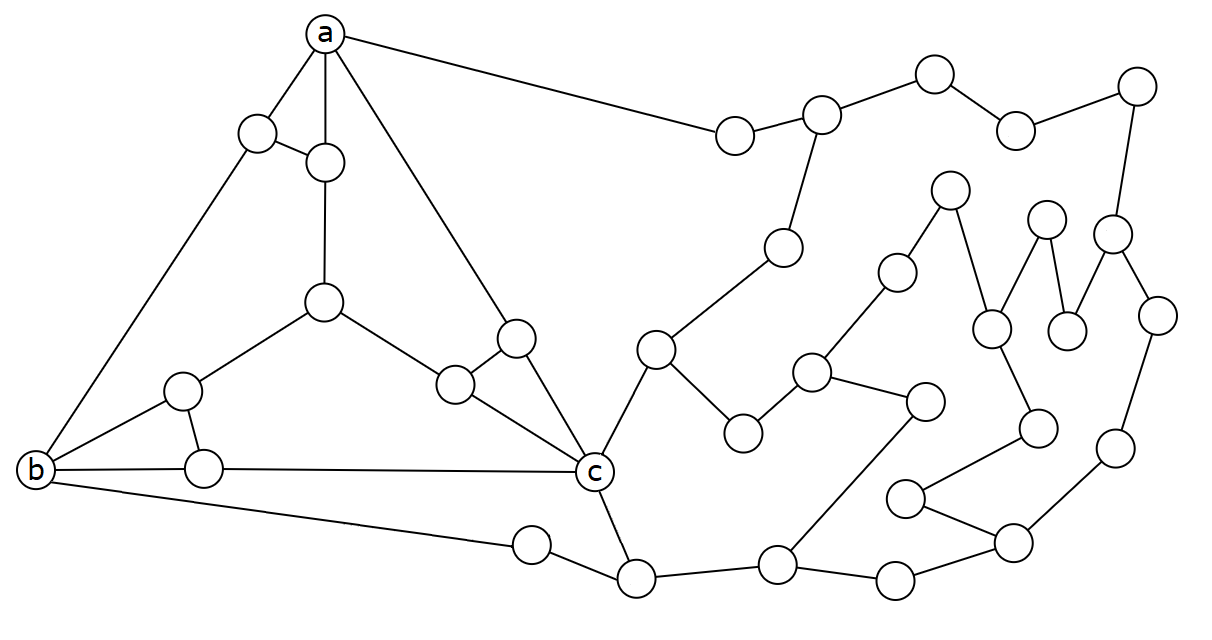}
				\includegraphics[width=.4\textwidth]{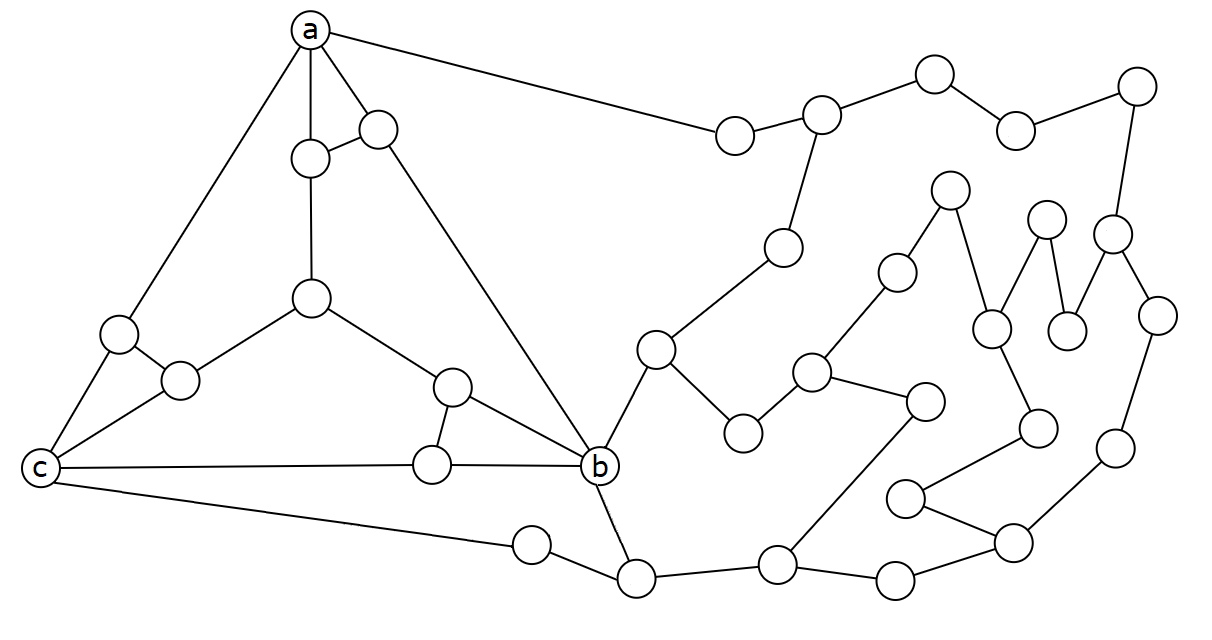}\\
				\textsc{Graph G} \hspace{3cm} \textsc{Graph H}
		\end{center}
	\end{figure}
\end{myeg}

\bibliography{Jacobian}\bibliographystyle{alpha}

\end{document}